\documentclass[reqno,b5paper]{amsart}
\usepackage{amsmath}
\usepackage{amssymb}
\usepackage{amsthm}
\usepackage{enumerate}
\usepackage[mathscr]{eucal}
\newtheorem{theorem}{Theorem}[section]
\newtheorem{remark}{Remark}[section]

\newtheorem{lemma}{Lemma}[section]

\newtheorem{corollary}{Corollary}[section]
\numberwithin{equation}{section}
\begin{document}

\title[Uncentered maximal function]{Uncentered maximal function for elliptic partial differential operator}

\author[C. Abdelkefi \and S. Chabchoub]{Chokri Abdelkefi* \and Safa Chabchoub**  }
\newcommand{\acr}{\newline\indent}
\address{\llap{*\,}
Department of Mathematics\acr Preparatory Institute of Engineer
Studies of Tunis \acr 1089 Monfleury Tunis, University of Tunis\acr
Tunisia} \email{chokri.abdelkefi@yahoo.fr}
\address{\llap{**\,} Department of Mathematics\acr Faculty of Sciences of Tunis\acr 1060
Tunis, University of Tunis El Manar\acr Tunisia}
\email{safachabchoub@yahoo.com}
\thanks{This work was completed with the support of the DGRST research project LR11ES11,
University of Tunis El Manar, Tunisia.}

\begin{abstract}
In the present paper, we study in the harmonic analysis associated
to the Weinstein operator, the boundedness on $L^p$ of the
uncentered maximal function. First, we establish estimates for the
Weinstein translation of characteristic function of a closed ball
 with
radius $\varepsilon$ centered at $0$ on the upper half space
$\mathbb{R}^{d-1}\times[ 0,+\infty[$. Second, we prove weak-type
$L^1$-estimates for the uncentered maximal function associated with
the Weinstein operator and we obtain that is bounded on $L^p$ for $1
< p \leq+\infty$.
\end{abstract}
\maketitle
\section{Introduction}
For a real parameter $\alpha>-\frac{1}{2}$ and $d\geq2$, the
Weinstein operator (also called Laplace-Bessel operator) is the
elliptic partial differential operator $\Delta_{d,\alpha}$ defined
on the upper half space $\mathbb{R}^d_{+}=\mathbb{R}^{d-1} \times\,
(0,+\infty)$ by
$$ \Delta_{d,\alpha} = \sum_{i=1}^d \frac{\partial^2}{\partial x_i^2}+ \frac{2\alpha+1}{x_d} \frac{\partial}{\partial x_d}.$$
The operator $\Delta_{d,\alpha}$ can be written as
$$\Delta_{d,\alpha} =\Delta_{d-1}+\mathcal{L}_\alpha,
$$ where $\Delta_{d-1} $ is the Laplacian operator on $
\mathbb{R}^{d-1}$ and $\mathcal{L}_\alpha $ is the Bessel operator
on $(0,+\infty)$ with respect to the variable $ x_d$ given by
$$\mathcal{L}_\alpha = \frac{\partial^2}{\partial x_d^2}+ \frac{2\alpha+1}{x_d}\frac{\partial}{\partial x_d}.$$
For $d>2$, the operator $\Delta_{d,\alpha}$ arises as the
Laplace-Beltrami operator on the Riemannian space $\mathbb{R}^d_{+}$
equipped with the metric $\displaystyle ds^2=
x_d^{\frac{4\alpha+2}{d-2}}\sum_{i=1}^{d} dx_i^2$. The Weinstein
operator $\Delta_{d,\alpha}$ has important applications in both pure
and applied mathematics, especially in the fluid mechanics (see
\cite{A}). Many authors were interested in the study of the
Weinstein equation $\Delta_{d,\alpha}u = 0$, one can cite for
instance M. Brelot \cite{M} and H. Leutwiler \cite{L}. The harmonic
analysis associated with the Weinstein operator was studied in
\cite{B,NZ}. In particular, the authors have introduced and studied
the generalized Fourier transform associated with the Weinstein
operator also called the Weinstein transform.

The Hardy-Littlewood maximal function was first introduced by Hardy
and Littlewood in 1930 for functions defined on the circle (see
\cite{GH}). Later it was extended to various Lie groups, symmetric
spaces, some weighted measure spaces (see \cite{J,GG,E,JO,TX,AT})
and different hypergroups (see \cite{W,WC,K}).

In this paper, we denote by $B^+(0,\varepsilon) =
B(0,\varepsilon)\cap \mathbb{R}^d_{+}$, the closed ball on
$\mathbb{R}^d_{+}$ with radius $\varepsilon$ centered at $0$. For $
x\in \mathbb{R}^d_{+},$ we establish estimates of the Weinstein
translation (see next section) of the characteristic function of
$B^+(0,\varepsilon)$,
$\tau_{x}(\chi_{B^+(0,\varepsilon)})(-y',y_d),$ based on the
inversion formula, where we put $y=(y',y_d)$ in $\mathbb{R}^d_{+}$
with $y'=(y_1,...,y_{d-1})$. Using these estimates, we prove the
 weak-type $(1, 1)$ of the uncentered maximal function $M_\alpha f$, defined
 for each integrable function $f$ on $(
\mathbb{R}^d_{+},\nu_\alpha)$ and $x\in \mathbb{R}^{d}_{+}$ by
$$\displaystyle M_\alpha f(x) = \sup_{\varepsilon > 0,\;z\in
B^+(x,\varepsilon)} \frac{1}{\nu_\alpha(B^+(0,\varepsilon))}
 \left|\int_{\mathbb{R}^d_{+}}
 f(y)
 \tau_{z}(\chi_{B^+(0,\varepsilon)})(-y',y_{d})d\nu_\alpha(y)\right|,$$
 where $\nu_\alpha$  is a
weighted Lebesgue measure associated with the Weinstein operator
(see next section)  and
   $B^+(x,\varepsilon) = B(x,\varepsilon)\cap \mathbb{R}^{d}_{+}$ the closed ball on
$\mathbb{R}^d_{+}$ with radius $\varepsilon$ centered at $x$.
Finally, we obtain the $L^p$-boundedness of $M_\alpha f$ when
 $1< p \leq+\infty$.

Bloom and Xu in \cite{WZ} have obtained analogous results for the
Ch\'{e}bli-Trim\`{e}che hypergroups. Later, similar results have
been established in \cite{Ab} for the harmonic analysis involving
the Dunkl operator on the real line.\\

 The contents of this paper are as follows.
 \\In section 2, we
collect some basic definitions and results about harmonic analysis
associated with Weinstein operator.\\In section 3, we establish in a
first step, estimates of  $\tau_{x}(\chi_{B^+(0,\varepsilon)})$, $x
\in\mathbb{R}^d_{+}$. In the second step, we prove that the
uncentered maximal function $M_\alpha f$ is of weak-type $(1, 1)$
and we obtain that is strong type $(p,p)$ for $1<p\leq+\infty$.

Along this paper, we denote $\langle .,.\rangle$ the usual Euclidean
inner product in $\mathbb{R}^d$ as well as its extension to
$\mathbb{C}^d \times\mathbb{C}^d$, we write for $x \in
\mathbb{R}^d,$ $\|x\| = \sqrt{\langle x,x\rangle}$. In the sequel
$c$ represents a suitable positive constant which is not necessarily
the same in each occurrence. Furthermore, we denote by

$\bullet\quad D_\ast( \mathbb{R}^{d})$ the space of
$C^\infty$-functions which are of compact support, even with respect
to the last variable.

$\bullet\quad S_\ast( \mathbb{R}^{d})$ the space of
$C^\infty$-functions which are rapidly decreasing together with
their derivatives, even with respect to the last variable.

\section{Preliminaries}
In this section, we recall some notations and results about harmonic
analysis associated with the Weinstein operator. \\For every $1 \leq
p \leq + \infty$, we denote by $L^{p}(\mathbb{R}^d_{+},\nu_\alpha)$
the space of measurable functions $f$ on $\mathbb{R}^d_{+}$ such
that
$$\|f\|_{p,\alpha} = \left(\int_{ \mathbb{R}^d_{+}}|f(x)|^p
d\nu_\alpha(x) \right)^{1/p} < + \infty,\quad \mbox{ if } p < +
\infty$$ and $$\|f\|_{\infty} = ess\sup_{x
\in\mathbb{R}^d_{+}}|f(x)| < + \infty,$$ where $\nu_\alpha $ is a
measure defined by
$$d\nu_\alpha(x)= \frac{x_{d}^{2\alpha+1}}{(2\pi)^{\frac{d-1}{2}} 2^{\alpha} \Gamma(\alpha+1)}\;dx
= \frac{x_{d}^{2\alpha+1}}{(2\pi)^{\frac{d-1}{2}}
2^{\alpha}\Gamma(\alpha+1)}\; dx_{1}...dx_{d}. $$ For a radial
function $f\in L^{1}(\mathbb{R}^d_{+},\nu_\alpha)$, the function $F$
defined on $\mathbb{R}_{+}$ such that $ f(x)= F(\|x\|)$, for all $
x\in \mathbb{R}^d_{+}$ is integrable with respect to the measure $\,
r^{2\alpha+d} dr$. More precisely, we have
\begin{eqnarray}
 \int_{ \mathbb{R}^d_{+}}f(x) d\nu_\alpha(x)=
\frac{1}{2^{\alpha+\frac{d-1}{2}} \Gamma(\alpha+\frac{d+1}{2})}
\int_{0}^{+\infty} F(r)r^{2\alpha+d} dr.
\end{eqnarray}
 For all $\lambda =(\lambda_{1},\lambda_{2},...,\lambda_{d}) \in
\mathbb{C}^{d}$, the system
$$
 \left\{
\begin{array}{ll}
 \frac{\partial^2 u}{\partial x_j^2}(x) &= - \lambda_j^2 u(x), \;  j=1,...,d-1, \\
 \mathcal{L}_\alpha u(x) &= - \lambda_d^2 u(x)\\
 u(0) &= 1,\;\frac{\partial u}{\partial x_d}(0)= 0,\;\frac{\partial u}{\partial x_j}(0)= - i \lambda_j,\; j=1,...,d-1,
\end{array}
\right.
$$
has a unique solution on $\mathbb{R}^d$, denoted by $\Psi_\lambda $
called the Weinstein kernel and given by
\begin{eqnarray}
 \Psi_\lambda (x) = e^{-i\langle x',\lambda'\rangle} j_\alpha(x_d \lambda_d).
\end{eqnarray}
Here $x=(x',x_d)\in \mathbb{R}^d_{+}$, $x'=(x_{1},...,x_{d-1}),
\lambda'=(\lambda_{1},...,\lambda_{d-1})$ and $j_\alpha$
 is the normalized
Bessel function of the first kind and order $\alpha$, defined by
\begin{eqnarray}
 j_\alpha(\lambda x) = \left\{ \begin{array}{ll}
2^\alpha \Gamma (\alpha +1) \,\frac{J_\alpha(\lambda x)}{(\lambda
x)^\alpha} &\mbox{ if } \lambda x \neq 0\\
1 &\mbox{if}\; \lambda x = 0,
\end{array}\right.
\end{eqnarray}
where $J_\alpha$ is the Bessel function of first kind and order
$\alpha$ (see \cite{GN}).\\
We have for all $x \in \mathbb{R}$, the
function $\lambda \rightarrow j_\alpha (\lambda x)$ is even on
$\mathbb{R}$.\\
The Weinstein kernel $ \Psi_\lambda (x)$ has a unique extension to $
\mathbb{C}^{d}\times \mathbb{C}^{d}$. It has the following
properties :
\begin{itemize}
\item[i)]$ \forall\; \lambda,z \in \mathbb{C}^{d},\; \Psi_\lambda (z)=\Psi_z (\lambda)$.
\item[ii)] $ \forall\; \lambda,z \in \mathbb{C}^{d},\; \Psi_\lambda (-z)=\Psi_{-\lambda}(z)$.
\item[iii)] $ \forall\; \lambda,x \in \mathbb{R}^{d}$
\begin{eqnarray}
 |\Psi_{\lambda} (x)|\leq 1.
\end{eqnarray}
\end{itemize}
There exists an analogue of the classical Fourier transform with
respect to the Weinstein kernel called the Weinstein transform and
denoted by $\mathcal{F}_{W}$. The Weinstein transform enjoys
properties similar to those of the classical Fourier transform and
is defined for $f \in L^{1}(\mathbb{R}^d_{+},\nu_\alpha)$ by
\begin{eqnarray}
\mathcal{F}_{W}(f)(\lambda) =
\int_{\mathbb{R}^{d}_{+}}f(y)\,\Psi_\lambda (y) \, d\nu_{\alpha}(y),
\quad \lambda \in \mathbb{R}^{d}_{+}.
\end{eqnarray}
 We list some known properties of this transform :
 \begin{itemize}
 \item[i)] For all $f \in L^{1}(\mathbb{R}^d_{+},\nu_{\alpha})$, we have
\begin{eqnarray}
 \|\mathcal{F}_{W}(f)\|_{\infty} \leq \|f\|_{1,\alpha}.
\end{eqnarray}

 \item[ii)] Let $f \in L^{1}(\mathbb{R}^d_{+},\nu_{\alpha})$. If $\mathcal{F}_{W}(f) \in L^1(\mathbb{R}^d_{+},\nu_{\alpha})$, then  we have the inversion formula
\begin{eqnarray}
 f(x) = \int_{\mathbb{R}^d_{+}} \Psi_{\lambda} (-x',x_{d}) \mathcal{F}_{W}(f)(\lambda)d\nu_{\alpha}(\lambda),\quad x \in \mathbb{R}^d_{+}.
\end{eqnarray}
\item[iii)]The Weinstein transform $ \mathcal{F}_{W}$ on $S_\ast( \mathbb{R}^{d})$ exends
uniquely to an isometric isomorphism on
$L^{2}(\mathbb{R}^d_{+},\nu_{\alpha}) $.
\item[iv)] Plancherel formula : For all $f \in
L^{2}(\mathbb{R}^d_{+},\nu_{\alpha})$, we have $$
\int_{\mathbb{R}^{d}_{+}}|f(x)|^{2}d\nu_{\alpha}(x)=
\int_{\mathbb{R}^{d}_{+}}|\mathcal{F}_{W}(f)(x)|^{2}d\nu_{\alpha}(x)
$$
\item[v)] Let $f \in L^{1}(\mathbb{R}^d_{+},\nu_{\alpha})$ be a radial function, then the function $F$ such that $f(x)=F(\|x\|)$
is integrable on $(0,+\infty)$ with respect to the measure
$r^{2\alpha+d}dr$ and its Weinstein transform is given for  $y\in
\mathbb{R}^{d}_{+},$ by
\begin{eqnarray}
  \mathcal{F}_{W}(f)(y) = \mathcal{F}_{B}^{\alpha+\frac{d-1}{2}}(F)(\|y\|),
\end{eqnarray}
where $ \mathcal{F}_{B}^{\gamma}$ is the Fourier-Bessel transform of
order $\gamma$, $\gamma > -\frac{1}{2}$, given by
$$\mathcal{F}_{B}^{\gamma}(F)(\lambda)=
\frac{1}{2^\gamma \Gamma(\gamma+1)}\int_{0}^{+\infty}
 F(r) j_{\gamma}(\lambda r) r^{2\gamma+1}dr,\quad \lambda \in (0,+\infty).$$
\end{itemize}
 For $x, y \in
\mathbb{R}^d_{+}$ and $f$ a continuous function on $\mathbb{R}^{d}$
which is even with respect to the last variable, the Weinstein
translation operator $\tau_{x}$ is given by
\begin{eqnarray}
 \tau_{x}(f)(y) = \int_{0}^{+\infty}f(x'+y',\rho) W_{\alpha}(x_{d},y_{d},\rho) \rho^{2\alpha+1}
  d\rho,
\end{eqnarray}
where  the kernel $W_{\alpha} $ is given by\\
 $W_{\alpha}(x_{d},y_{d},\rho)$
\begin{eqnarray}
=\frac{\Gamma(\alpha+1)((x_d+y_d)^{2}-\rho^{2})^{\alpha-\frac{1}{2}}
(\rho^{2}-(x_d-y_d)^{2})^{\alpha-\frac{1}{2}}}{
2^{2\alpha-1}\sqrt{\pi}\;
\Gamma(\alpha+\frac{1}{2})(x_{d}y_{d}\rho)^{2\alpha}}
\chi_{]|x_{d}-y_{d}|, x_d+y_d[}(\rho).
 \end{eqnarray} For all $x_{d},y_{d} > 0$, we have
\begin{eqnarray}
\int_0^{+\infty}W_\alpha(x_d,y_d,\rho)\rho^{2\alpha+1}
  d\rho=1.
\end{eqnarray}
The Weinstein translation operator satisfies the following
properties.
\begin{itemize}

\item[i)]For all continuous function $f$ on $\mathbb{R}^d $ which is  even
with respect to the last variable and $x,y \in \mathbb{R}^d_{+}$, we
have
$$\tau_{x}(f)(y)= \tau_{y}(f)(x) , \quad \tau_{0}f=f.$$
\item[ii)] For all $f \in S_\ast( \mathbb{R}^{d})$ and $y \in \mathbb{R}^d_{+}$,
 the function $x\rightarrow \tau_{x}f(y) $ belongs to $S_\ast( \mathbb{R}^{d})$.
\item[iii)]For all $f \in L^{p}(\mathbb{R}^{d}_{+},\nu_{\alpha})$, $1\leq p \leq +
\infty$ and $x \in \mathbb{R}^d_{+}$, we have
\begin{eqnarray}
 \|\tau_{x}f\|_{p,\alpha} \leq \|f\|_{p,\alpha}.
\end{eqnarray}

\end{itemize}
For a function $f \in L^{p}(\mathbb{R}^d_{+},\nu_\alpha)$, $p=1$ or
$2$ and $ x \in \mathbb{R}^d_{+}$, the Weinstein translation
$\tau_{x}$ is also defined by the following relation:
\begin{eqnarray}
 \mathcal{F}_{W}(\tau_{x}f)(\lambda)= \Psi_{\lambda} (-x', x_{d})\mathcal{F}_{W}(f)(\lambda) , \quad \lambda \in \mathbb{R}^d_{+}.
\end{eqnarray}
By using the Weinstein translation, we define the convolution
product $f\,\ast_{\alpha}\, g$ of functions $f,g \in
L^1(\mathbb{R}^d_{+},\nu_{\alpha})$ as follows:
$$(f\,\ast_\alpha\, g)(x) = \int_{\mathbb{R}^{d}_{+}} \tau_{x}(f)(-y',y_{d}) g(y)
d\nu_{\alpha}(y),\quad x \in \mathbb{R}^{d}_{+}. $$ This convolution
is commutative and associative and satisfies the following results.
\begin{enumerate}
\item[i)] Let  $1\leq p,q, r \leq + \infty$ such that
$\frac{1}{p} + \frac{1}{q} = 1 + \frac{1}{r}$ (the Young condition).
If  $f \in L^p(\mathbb{R}^d_{+},\nu_\alpha)$ and $g \in
L^q(\mathbb{R}^d_{+},\nu_\alpha)$, then $f\,\ast_\alpha\, g \in
L^r(\mathbb{R}^d_{+},\nu_\alpha)$ and we have
\begin{eqnarray}\|f\,\ast_\alpha\,g\|_{r,\alpha}
\leq \|f\|_{p,\alpha} \|g\|_{q,\alpha}.\end{eqnarray}
\item[ii)] For all $f \in L^{1}(\mathbb{R}^d_{+},\nu_\alpha)$ and $g \in
L^{p}(\mathbb{R}^d_{+},\nu_\alpha)$, $p=1$ or $2$, we have
  \begin{eqnarray}  \mathcal{F}_{W}(f\,\ast_\alpha g) =
\mathcal{F}_{W}(f) \mathcal{F}_{W}(g).\end{eqnarray}
 \end{enumerate}
\section{Weak-type (1.1) of the uncentered maximal function}
In this section, we establish estimates of $\;
\tau_{x}(\chi_{B^+(0,\varepsilon)})(-y',y_d)$, $x,y \in
\mathbb{R}^d_{+}$ and we prove the weak-type $(1,1)$ of the
uncentered maximal function $M_\alpha f$ and we obtain that is
bounded on $L^{p}$ for $1 <p\leq+\infty$.
\\

The following remark plays a key role.
\begin{remark}
 For any $x, y \in \mathbb{R}^d_{+}$ and $\varepsilon > 0$,
 we have\begin{eqnarray} \tau_{x}(\chi_{B^+(0,\varepsilon)})(-y',y_d)=\int_0^{+\infty}\chi_{B^+(0,\varepsilon)}(x'-y',\rho) W_\alpha(x_d,y_d,\rho) \rho^{2\alpha+1}
  d\rho.\end{eqnarray}
  Put  $u=(x'-y',\rho)$, we have $\displaystyle \|u\|^{2}=\sum_{i=1}^{d-1}(x_i-y_i)^2+\rho^2$. Then by
  (2.9) and (2.10), we have
  \begin{eqnarray}
 \sqrt{\sum_{i=1}^{d-1}(x_i-y_i)^2+(x_d-y_d)^2} < \|u\| <\sqrt{\sum_{i=1}^{d-1}(x_i-y_i)^2+(x_d+y_d)^2}.
  \end{eqnarray}
  From (3.1), we have $ u\in B^{+}(0,\varepsilon),$ which gives according to
  (3.2),\\
$\tau_{x}(\chi_{B^+(0,\varepsilon)})(-y',y_{d})=0$ when $\|x-y\|=
 \displaystyle\sqrt{\sum_{i=1}^{d-1}(x_i-y_i)^2+(x_d-y_d)^2}\geq \varepsilon.$
 Then we can assume that $y \in \mathbb{R}^d_{+}$ satisfies $ \|x-y\| <
 \varepsilon.$ Note that $ \|x-y\| <
 \varepsilon$ implies $|x_{d}-y_{d}| < \varepsilon$.
\end{remark}
\begin{lemma}
Let $\lambda \in \mathbb{R}^d_{+}$ and $\varepsilon \in ]0,+\infty[
$, then we have \begin{eqnarray} &i)&
|\mathcal{F}_{W}(\chi_{B^+(0,\varepsilon)})(\lambda)| \leq  c \;
  \varepsilon^{2\alpha+d+1},\\
&ii)&|\mathcal{F}_{W}(\chi_{B^+(0,\varepsilon)})(\lambda)| \leq c\;
\varepsilon^{\alpha +\frac{d}{2}}
\|\lambda\|^{-(\alpha+\frac{d}{2}+1)}.
\end{eqnarray}
Here $c$ is a constant which depends only on $\alpha$ and $d$.
\end{lemma}
\begin{proof}
By (2.8), we can write for $
 \lambda \in \mathbb{R}^d_{+}$ and $\varepsilon \in (0,+\infty),$\\
$\mathcal{F}_{W}(\chi_{B^+(0,\varepsilon)})(\lambda)$
\begin{eqnarray}
 &=&\frac{1}{2^{\alpha+\frac{d-1}{2}} \Gamma(\alpha+\frac{d+1}{2})}
 \int_0^{+\infty} \chi_{B^+(0,\varepsilon)}(r) j_{\alpha+\frac{d}{2}-\frac{1}{2}}(\|\lambda\|r) r^{2\alpha+d} dr \nonumber\\
  &=& \frac{\varepsilon^{2\alpha+d+1}}{2^{\alpha+\frac{d+1}{2}} \Gamma(\alpha+\frac{d+3}{2})} \;
 j_{\alpha+\frac{d}{2}+\frac{1}{2}}(\|\lambda\| \varepsilon).\end{eqnarray}
 Since
 $| j_{\alpha+\frac{d}{2}+\frac{1}{2}}(\|\lambda\| \varepsilon)| \leq 1$, we get
 \begin{eqnarray*}
  |\mathcal{F}_{W}(\chi_{B^+(0,\varepsilon)})(\lambda)| \leq c \;
  \varepsilon^{2\alpha+d+1}.
 \end{eqnarray*}
 Now, from (2.3), (3.5) and the fact that the function $z \mapsto
\sqrt{z}J_\alpha(z)$ is bounded on $(0,+\infty)$, we can see that
\begin{eqnarray*}
|\mathcal{F}_{W}(\chi_{B^+(0,\varepsilon)})(\lambda)|&=&
 \frac{\varepsilon^{2\alpha+d+1}}{2^{\alpha+\frac{d+1}{2}} \Gamma(\alpha+\frac{d+3}{2})}
| j_{\alpha+\frac{d}{2}+\frac{1}{2}}(\|\lambda\|
\varepsilon)|\nonumber\\ &=& \frac{1}{2^{\alpha+\frac{d+1}{2}}
\Gamma(\alpha+\frac{d+3}{2})} \;\frac{\varepsilon^{\alpha +
\frac{d}{2}}}{\|\lambda\|^{\alpha+\frac{d}{2}+1} }\sqrt{\|\lambda\|
\varepsilon}\,|J_{\alpha+\frac{d}{2}+\frac{1}{2}} (\|\lambda\|
\varepsilon)|\nonumber\\ &\leq& c\; \varepsilon^{\alpha
+\frac{d}{2}} \|\lambda\|^{-(\alpha+\frac{d}{2}+1)}.\end{eqnarray*}
Hence the lemma is proved .
\end{proof}
\begin{lemma}
 For $\alpha > \frac{d}{2}-1$, there exists $c > 0$ such that
 for any $x \in \mathbb{R}^d_{+}$ with $x_{d}> 0$ and $\varepsilon > 0$,
 we have
\begin{eqnarray}
  0 \leq \tau_{x}(\chi_{B^+(0,\varepsilon)})(-y',y_d) \leq c \; \Big(\frac{\varepsilon}{x_d}\Big)^{2\alpha+1},\quad a.e \,y \in \mathbb{R}^d_{+} .
\end{eqnarray}
Here $c$ is a constant which depends only on $\alpha$ and $d$.
\end{lemma}
\begin{proof}
Let $x \in \mathbb{R}^d_{+}$ and $\varepsilon > 0$. Using (2.9) and
(2.11) we have
 \begin{eqnarray}
  0\leq \tau_{x}(\chi_{B^+(0,\varepsilon)})(-y',y_d)\leq 1, \quad a.e\,y \in \mathbb{R}^d_{+}.
 \end{eqnarray}
If $0< x_d \leq 2\varepsilon$, we obtain that
\begin{eqnarray*}
 1 \leq \Big(\frac{2\varepsilon}{x_d}\Big)^{2\alpha+1},
\end{eqnarray*}
hence, by (3.7) we deduce (3.6).\\
Therefore we can assume in the following argument that $x_{d} >
2\varepsilon,$ and  in view of Remark 3.1 that $y\in
\mathbb{R}^d_{+} $ satisfies $|x_d - y_d| < \varepsilon$.\\Take
$\psi \in D_\ast( \mathbb{R}^{d}) $
 satisfying $0 \leq \psi(x) \leq 1$, supp $\psi \subset B^{+}(0,1)$ and
 $\|\psi\|_{1,\alpha} = 1$. Put
 $$\psi_t(x) = \frac{1}{t^{2\alpha+d+1}}\;\; \psi\big(\frac{x}{t}\big),\;\;
 t > 0,\; x \in \mathbb{R}^{d}_{+}, $$ the dilation of $\psi.$
We have $\psi_t \in D_\ast(\mathbb{R}^d)$ which gives
$\mathcal{F}_{W}(\psi_t)\in S_\ast(\mathbb{R}^d),$ then we can
assert that both of $\tau_x(\chi_{B^+(0,\varepsilon)}
 \ast_\alpha\, \psi_t)$ and
$\mathcal{F}_{W}(\tau_x(\chi_{B^+(0,\varepsilon)}
 \ast_\alpha\, \psi_{t}))$ are in $L^{1}(\mathbb{R}^d_{+},\nu_\alpha)$. Using (2.7), (2.13) and (2.15), we obtain for
 $y\in \mathbb{R}^d_{+}$ \\
 $ \tau_x(\chi_{B^+(0,\varepsilon)}\ast_{\alpha}\,\psi_{t})(y)$
 \begin{eqnarray}
  &=& \int_{\mathbb{R}_+^{d}}
 \Psi_{\lambda} (-x',x_{d})\Psi_\lambda (-y',y_{d}) \mathcal{F}_{W}
 (\chi_{B^+(0,\varepsilon)})(\lambda) \mathcal{F}_{W}(\psi_t)(\lambda)d\nu_{\alpha}(\lambda)\nonumber\\&=&
 \int_{\mathbb{R}_+^{d}} e^{i[\langle x',\lambda'\rangle+\langle
y',\lambda'\rangle]} j_{\alpha}(\lambda_{d} x_{d})
j_{\alpha}(\lambda_{d} y_{d}) \mathcal{F}_{W}
 (\chi_{B^{+}(0,\varepsilon)})(\lambda) \mathcal{F}_{W}(\psi_{t})(\lambda)
d\nu_{\alpha}(\lambda).\nonumber\\
 \end{eqnarray}
Clearly we have $\|\psi_{t}\|_{1,\alpha} = 1$. According to (2.6),
we have
  \begin{eqnarray}
   |\mathcal{F}_{W}(\psi_{t})(\lambda)| \leq \|\mathcal{F}_{W}(\psi_t)\|_{\infty}
  \leq \|\psi_{t}\|_{1,\alpha} = 1,\quad a.e \, \lambda
   \in  \mathbb{R}^d_{+}.
  \end{eqnarray}
   Let us decompose (3.8) as a sum of three terms:
  \begin{eqnarray}
   \qquad\qquad \tau_{x}(\chi_{B^{+}(0,\varepsilon)}\ast_{\alpha}\,\psi_{t})(y)
   &=&  \int_{\|\lambda\|\leq x_{d}^{-1}} +
   \int_{x_{d}^{-1} \leq \|\lambda\| \leq \varepsilon^{-1}} + \int_{\varepsilon^{-1}\leq
   \|\lambda\|} \nonumber\\
   &=& I_{1} + I_{2} + I_{3} .
  \end{eqnarray}
  From (2.1), (2.4), (3.3) and (3.9), we obtain
\begin{eqnarray}
\qquad\qquad\qquad  |I_{1}| &\leq& c \; \varepsilon^{2\alpha+d+1}
\int_{ \|\lambda\| \leq x_d^{-1} } d\nu_\alpha(\lambda)\nonumber\\
&\leq& c \; \Big(\frac{\varepsilon}{x_{d}}\Big)^{2\alpha+d+1}\nonumber\\
&\leq& c \; \Big(\frac{\varepsilon}{x_{d}}\Big)^{2\alpha+1},\quad
\mbox{for}\; x_{d} > 2\varepsilon.
\end{eqnarray}
To estimate $I_2$, we observe that for $x_{d}> 2\varepsilon $ and
$|x_{d}-y_{d}|< \varepsilon $, we
 have
$$\frac{1}{2} x_{d} < x_{d} - \varepsilon < y_{d} < x_{d} + \varepsilon <
\frac{3}{2} x_{d},$$ so we deduce
\begin{eqnarray}
 0 < y_{d}^{-(\alpha+\frac{1}{2})}
< c\; x_{d}^{-(\alpha+\frac{1}{2})}.
\end{eqnarray}
By (2.3) and the fact that the function $z \mapsto
\sqrt{z}J_\alpha(z)$ is bounded on $(0,+\infty)$, we can write
\begin{eqnarray}
 |j_{\alpha}(z)| \leq c\;
z^{-(\alpha+\frac{1}{2})},
\end{eqnarray}
 then from (2.1), (3.3), (3.9), (3.12) and (3.13), we get
\begin{eqnarray}|I_{2}| &\leq& c\;
\varepsilon^{2\alpha+d+1}
x_d^{-(\alpha+\frac{1}{2})}y_{d}^{-(\alpha+\frac{1}{2})}
 \int_{x_{d}^{-1} \leq \|\lambda\| \leq \varepsilon^{-1}}\lambda_{d}^{-(2\alpha+1)}
d\nu_\alpha(\lambda)\nonumber\\
&\leq& c\; \varepsilon^{2\alpha+d+1} x_{d}^{-2\alpha-1}
\big(\frac{1}{\varepsilon^{d}} - \frac{1}{x_{d}^{d}}\big) \nonumber\\
&\leq&c\;
\varepsilon^{2\alpha+1}x_{d}^{-2\alpha-1} \nonumber \\
& \leq& c\; \Big(\frac{\varepsilon}{x_d}\Big)^{2\alpha+1} ,\quad
\mbox{for}\; x_{d} > 2\varepsilon.
\end{eqnarray}
For $I_{3}$, we use (2.1), (3.4), (3.9), (3.12) and (3.13) and we
find that
\begin{eqnarray*}
|I_3| &\leq& c \;\varepsilon^{\alpha + \frac{d}{2}}x_d^{-(\alpha+\frac{1}{2})}y_d^{-(\alpha+\frac{1}{2})}
 \int_{\varepsilon^{-1}\leq \|\lambda\|}\|\lambda\|^{-(\alpha+\frac{d}{2}+1)}\lambda_{d}^{-(2\alpha+1)}
d\nu_\alpha(\lambda)\\ &\leq& c\; \varepsilon^{\alpha +
\frac{d}{2}}x_{d}^{-2\alpha-1} \int^{+
\infty}_{\varepsilon^{-1}}r^{-\alpha+\frac{d}{2}-2} dr.
\end{eqnarray*}
Since $\alpha > \frac{d}{2}-1$, we obtain
\begin{eqnarray}
 |I_3| &\leq& c\; \varepsilon^{\alpha + \frac{d}{2}} x_d^{-(2\alpha+1)} \varepsilon^{\alpha -\frac{d}{2}+1}\nonumber \\
 &\leq& c \;\Big(\frac{\varepsilon}{x_d}\Big)^{2\alpha+1},\quad
\mbox{for}\; x_{d} > 2\varepsilon.
\end{eqnarray}
Thus we get by (3.10), (3.11), (3.14) and (3.15)
\begin{eqnarray*}
0 \leq\tau_x(\chi_{B^+(0,\varepsilon)}\ast_\alpha\,\psi_t)(-y',y_d)
\leq c\; \Big(\frac{\varepsilon}{x_d}\Big)^{2\alpha+1},\quad
\mbox{for}\; x_{d} > 2\varepsilon.
\end{eqnarray*}
Now using (2.9) and Fatou's Lemma, we can assert that
$$\lim_{t\rightarrow 0^+} \tau_{x}(\chi_{B^+(0,\varepsilon)}\ast_\alpha\,\psi_t)(-y',y_d) =
\tau_x(\chi_{B^+(0,\varepsilon)})(-y',y_d),\quad a.e\, y \in
  \mathbb{R}^d_{+}.$$ Hence, we deduce that $$0 \leq \tau_x
(\chi_{B^+(0,\varepsilon)})(-y',y_d) \leq c\;
\Big(\frac{\varepsilon}{x_d}\Big)^{2\alpha+1},$$ $\;\mbox{ for }
x_{d}> 2\varepsilon , \;a.e\, y \in
  \mathbb{R}^d_{+}\; with\; |x_d-y_d| < \varepsilon$,\; therefore (3.6) is
established.
\end{proof}
\noindent\textbf{Notation :} For $x \in \mathbb{R}^d_{+} $ and
$\varepsilon
> 0 $, we put $$C^+(x,\varepsilon)= B_{d-1}(x',\varepsilon)\times
]\max\{0,x_d-\varepsilon\},x_d+\varepsilon[,$$ with $x=(x',x_{d})$
and $B_{d-1}(x',\varepsilon) $ is the closed ball on
$\mathbb{R}^{d-1} $ with radius $\varepsilon$ centred at $x'$.
\begin{lemma}
 For $\alpha > \frac{d}{2}-1$, there exists $c > 0$ such that
 for any $x \in \mathbb{R}^d_{+}$ and $\varepsilon > 0$, we have
 \begin{eqnarray}
  0 \leq \tau_{x}(\chi_{B^+(0,\varepsilon)})(-y',y_d) \leq c\; \frac{\nu_\alpha(B^+(0,\varepsilon))}{\nu_\alpha(C^+(x,\varepsilon))},\quad a.e\, y \in
  \mathbb{R}^d_{+}.
 \end{eqnarray}
 Here $c$ is a constant which depends only on $\alpha$ and $d$.
\end{lemma}
\begin{proof}
Let $x \in \mathbb{R}^d_{+}$ and $\varepsilon > 0$. Using (2.1), we
have
\begin{eqnarray}
\nu_\alpha(B^+(0,\varepsilon))=
\frac{\varepsilon^{2\alpha+d+1}}{2^{\alpha+\frac{d-1}{2}}(2\alpha+d+1)
\Gamma(\alpha+\frac{d+1}{2})}\;.
 \end{eqnarray}
 On the one hand, we get for $x_{d} \leq \varepsilon,$
 $$C^+(x,\varepsilon)= B_{d-1}(x',\varepsilon)\times [0,x_d+\varepsilon[,$$
 then, we obtain
  \begin{eqnarray*}
   \nu_\alpha(C^+(x,\varepsilon))
&=&\frac{1}{(2\pi)^{\frac{d-1}{2}} 2^{\alpha} \Gamma(\alpha+1)}
 \int_{B_{d-1}(x',\varepsilon)} dy_1...dy_{d-1}
\int_0^{x_d+\varepsilon}y_{d}^{2\alpha+1} dy_d\\ &\leq & c\;
\varepsilon^{2(\alpha+1)} \int_{B_{d-1}(x',\varepsilon)}
dy_1...dy_{d-1} \leq c \;\varepsilon^{2\alpha+d+1}.
  \end{eqnarray*}
Using (3.17), we deduce
$$ \nu_\alpha(C^+(x,\varepsilon)) \leq  c\; \nu_\alpha(B^+(0,\varepsilon)),$$
  then by (3.7), we obtain (3.16) for $x_{d} \leq
\varepsilon$.
\\On the other hand, we have for
$x_{d} > \varepsilon,$  $$\;C^+(x,\varepsilon)=
B_{d-1}(x',\varepsilon)\times [x_d-\varepsilon,x_d+\varepsilon[,$$
 then , we obtain
\begin{eqnarray*}
   \nu_\alpha(C^+(x,\varepsilon))& =& \frac{1}{(2\pi)^{\frac{d-1}{2}} 2^{\alpha} \Gamma(\alpha+1)}
\int_{B_{d-1}(x',\varepsilon)} dy_1...dy_{d-1} \int_{x_d-\varepsilon}^{x_d+\varepsilon} y_d^{2\alpha+1}dy_{d}\\
&\leq & c \; \varepsilon^{d-1} \;(x_d+\varepsilon)^{2\alpha+1}
\times \varepsilon
\\&\leq & c \; \varepsilon^{d} \;x_d^{2\alpha+1}.
  \end{eqnarray*}
  Using (3.17), we get
  $$ \nu_\alpha(C^+(x,\varepsilon)) \leq  c\; \nu_\alpha(B^+(0,\varepsilon))
\Big(\frac{x_d}{\varepsilon}\Big)^{2\alpha+1},$$
  then by (3,6), we obtain (3.16) for $x_{d} > \varepsilon$, which proves the result.
\end{proof}
According to (\cite{E}, Lemma 1.6), we have the following Vitali
covering lemma.
\begin{lemma}
  Let $E$ be
a measurable subset of $\mathbb{R}^d_{+}$ (with respect to
$\nu_\alpha$) which is covered by the union of a family
$\{B_{j}^{+}\}$ where $B_{j}^{+} = B^{+}(x_j, r_j)$. Then from this
family we can select a  subfamily, $B_{1}^{+}, B_{2}^{+},...$ (which
may be finite) such that $ B_{i}^{+}\cap B_{j}^{+}= \emptyset $ for
$i \neq j $ and
$$\sum_{h} \nu_{\alpha} (B_{h}^{+}) \geq c\; \nu_\alpha(E).$$
\end{lemma}
 We recall that for $x \in \mathbb{R}^d_{+},$ $$ M_\alpha f(x) = \sup_{\varepsilon > 0,\;z\in B^+(x,\varepsilon)}
\frac{1}{\nu_\alpha(B^+(0,\varepsilon))}
 \left|\int_{\mathbb{R}^d_{+}}
 f(y)
 \tau_{z}(\chi_{B^+(0,\varepsilon)})(-y',y_{d})d\nu_\alpha(y)\right|,$$
so, we can write also  $$ M_\alpha f(x) = \sup_{\varepsilon
> 0,\;z\in B^+(x,\varepsilon)}
\frac{1}{\nu_\alpha(B^+(0,\varepsilon))}
 \Big|f\,\ast_\alpha\chi_{B^+(0,\varepsilon)}(z)\Big|.$$
\begin{theorem}
 The uncentered
maximal function  $M_{\alpha} f$ is of weak type (1,1).
\end{theorem}
\begin{proof}
 Let $\varepsilon > 0,\; x \in
\mathbb{R}^d_{+},\;z \in B^{+}(z,\varepsilon)$ and $f \in
L^{1}(\mathbb{R}^d_{+},\nu_\alpha)$. Using Remark 3.1, we have
\begin{eqnarray*}
\Big|f\,\ast_\alpha\chi_{B^+(0,\varepsilon)}(z)\Big|
&\leq&\int_{B^{+}(z,\varepsilon)} |f(y)|
\tau_{z}(\chi_{B^{+}(0,\varepsilon)})(-y',y_{d})d\nu_{\alpha}(y).
\end{eqnarray*}
By (3.16), we obtain
\begin{eqnarray*}
\Big|f\,\ast_\alpha\chi_{B^+(0,\varepsilon)}(z)\Big| &\leq& c\;\,
\frac{\nu_{\alpha}(B^{+}(0,\varepsilon))}{\nu_\alpha(C^{+}(z,\varepsilon))}\int_{B^{+}(z,\varepsilon)}
|f(y)| d\nu_{\alpha}(y).
\end{eqnarray*}
Since,  $B^{+}(z,\varepsilon)\subset C^{+}(z,\varepsilon) $, then

$$\Big|f\,\ast_\alpha\chi_{B^{+}(0,\varepsilon)}(z)\Big| \leq c \frac{\nu_\alpha(B^{+}(0,\varepsilon))}{\nu_{\alpha}(B^{+}(z,\varepsilon))}\int_{B^{+}(z,\varepsilon)} |f(y)| d\nu_{\alpha}(y).$$
Hence we deduce that
\begin{eqnarray}
 M_{\alpha} f (x) \leq
c\;\tilde{M}_\alpha f(x),
\end{eqnarray}
where $\tilde{M}_{\alpha} f $ is defined by
$$\tilde{M}_\alpha f(x)=\sup_{\varepsilon > 0,\;z \in B^{+}(x,\varepsilon)}
\frac{1} {\nu_{\alpha}\big(B^{+}(z,\varepsilon)\big)}  \int_{
B^{+}(z,\varepsilon)}|f(y)|d\nu_\alpha(y).$$ For $\lambda > 0$, put
$$\tilde{E}_{\lambda} = \{x \in \mathbb{R}^{d}_{+};\; \tilde{M}_{\alpha} f(x) >
\lambda \}.$$ Then, for each $x \in \tilde{E}_\lambda$, there exists
$\varepsilon >0 $ and $z\in B^{+}(x,\varepsilon),$ such that
\begin{eqnarray}
\int_{ B^{+}(z,\varepsilon)} |f(y)| d\nu_{\alpha}(y) > \lambda\;
\nu_{\alpha}\big(B^{+}(z,\varepsilon)\big).
\end{eqnarray}
Furthermore, note that $x\in B^{+}(z,\varepsilon)$, then when $x$
runs through  the set $\tilde{E}_{\lambda}$, the union of the
corresponding $B^{+}(z,\varepsilon)$ covers $\tilde{E}_{\lambda}$.
Thus using Lemma 3.4, we can select a disjoint subfamily $
B^{+}(z_{1},\varepsilon_{1}), B^{+}(z_{2},\varepsilon_{2}),...$
(which may be finite) such that
\begin{eqnarray}
 \sum_{h} \nu_{\alpha}\big(B^{+}(z_{h},\varepsilon_{h})\big)\geq
c\; \nu_{\alpha}(\tilde{E}_{\lambda}).
\end{eqnarray}
We have
\begin{eqnarray*}
 \Big(\int_{y\in \displaystyle \cup B^{+}(z_{h},\varepsilon_{h})} |f(y)| d\nu_{\alpha}(y)\Big)
= \sum_{h}\Big(\int_{y \in B^{+}(z_{h},\varepsilon_{h})} |f(y)|
d\nu_{\alpha}(y)\Big),
\end{eqnarray*}
 applying (3.19) and (3.20) to each of the mutually disjoint subfamily, we get
\begin{eqnarray*}
 \Big(\int_{y \in \displaystyle \cup B^{+}(z_{h},\varepsilon_{h})} |f(y)| d\nu_{\alpha}(y)\Big) >
\lambda \sum_{h} \nu_{\alpha}\big(B^{+}(z_{h},\varepsilon_{h})\big)
\geq \lambda\; c\; \nu_{\alpha}(\tilde{E}_{\lambda}).
\end{eqnarray*}
 But since the first member
of this inequality is  majorized by $\|f\|_{1,\alpha}$, we obtain
$$\nu_\alpha (\tilde{E}_{\lambda}) \leq
 c \;\frac{\|f\|_{1,\alpha}}{\lambda},$$ which gives that $\tilde{M}_{\alpha}f$ is of weak type $( 1,
1)$ and hence from (3.18) the same is true for $M_{\alpha}f$.
\end{proof}
As consequence of Theorem 3.1, we obtain the following corollary.
\begin{corollary}
 If $1 < p \leq + \infty \; \mbox{and} \; f \in L^{p}(\mathbb{R}^d_{+},\nu_{\alpha}),$ then we
have
$$M_{\alpha}f \in L^{p}(\mathbb{R}^d_{+}, \nu_{\alpha}) \quad  \mbox{and} \quad
\|M_{\alpha}f\|_{p,\alpha} \leq c\; \|f\|_{p,\alpha}.$$
\end{corollary}
\begin{proof} Using Theorem 3.1, (\cite{He}, Corollary 21.72) and
proceeding in the same manner as in the proof on Euclidean spaces
(see for example Theorem 1 in \cite{E}, section 1.3), we obtain the
results.

 \end{proof}

\end{document}